\documentclass[12pt, hidelinks]{article}
\usepackage[utf8]{inputenc}
\usepackage{fullpage}
\usepackage{amsmath}
\usepackage{url, float, graphicx, algorithm, algorithmic, placeins}
\usepackage{color, xcolor}
\usepackage{siunitx}
\usepackage{booktabs}
\usepackage[small,bf]{caption}

\newcommand{\reals}{{\mbox{\bf R}}}

\newcommand{\symm}{{\mbox{\bf S}}}  

\newcommand{\eg}{{\it e.g.}}
\newcommand{\ie}{{\it i.e.}}

\newcommand{\BEAS}{\begin{eqnarray*}}
\newcommand{\EEAS}{\end{eqnarray*}}
\newcommand{\BEA}{\begin{eqnarray}}
\newcommand{\EEA}{\end{eqnarray}}
\newcommand{\BEQ}{\begin{equation}}
\newcommand{\EEQ}{\end{equation}}
\newcommand{\BIT}{\begin{itemize}}
\newcommand{\EIT}{\end{itemize}}

\usepackage{amsthm}

\newtheorem{theorem}{Theorem}

\newcommand{\normsq}[1]{\left\|{#1}\right\|_2^2}

\newcommand{\norm}[1]{\left\|{#1}\right\|_2}
\newcommand{\lagr}{\mathcal{L}}

\newcommand{\umax}{u^\mathrm{max}}
\newcommand{\ellipse}{\mathcal{E}}
\newcommand{\eps}{\varepsilon}
\newcommand{\ito}[1]{1, \dots, {#1}}

\newcommand{\xcurr}{x^\mathrm{c}}
\newcommand{\xgoal}{x^\mathrm{g}}

\title{Fast Reciprocal Collision Avoidance Under Measurement Uncertainty}
\author{Guillermo Angeris\thanks{These authors contributed equally to this work.} \\ \texttt{\small angeris@stanford.edu} \\
    \and Kunal Shah\footnotemark[1] \\ \texttt{\small k2shah@stanford.edu}
    \and Mac Schwager \\ \texttt{\small schwager@stanford.edu}
\vspace{.2in}}

\date{May 2019}
\begin{document}

\maketitle

\begin{abstract}
We present a fully distributed collision avoidance algorithm based on convex
optimization for a team of mobile robots. This method addresses the practical 
case in which agents sense each other via measurements from noisy on-board 
sensors with no inter-agent communication. Under some mild conditions, we 
provide guarantees on mutual collision avoidance for a broad 
class of policies including the one presented. Additionally, we provide 
numerical examples of computational performance and show that, in both 2D and 
3D simulations, all agents avoid each other and reach their desired goals in 
spite of their uncertainty about the locations of other agents.
\end{abstract}

\section{Introduction}
Reliable collision avoidance is quickly becoming a mainstay requirement of any
scalable mobile robotics system. As robots continue to be deployed around
humans, assurances of safety become more critical, especially in high traffic
areas such as factory floors and hospital corridors. We present an on-line,
distributed collision avoidance algorithm based on convex optimization that
generates robot controls to evade moving obstacles sensed using noisy on-board
sensors. We also show that a general class of controllers, including the
one presented, guarantees mutual collision avoidance, provided that all robots
involved use this policy.  

We allow for each robot to have its own estimate of the relative positions of
other robots, which may be inconsistent with the other robots' estimates.  To
conservatively manage uncertainty in this model, we assume that each robot keeps
an uncertainty set (\eg, unions and intersections of ellipsoids) that contain
other robots' possible locations. We assume each robot knows its own position
exactly and updates its estimates of the other robots via noisy on-board sensors
such as a camera or LIDAR.

The policy is distributed in the sense that each robot \emph{only} requires an
estimate of the relative positions of the other robots. In other words, robots
do not need to communicate their position or explicitly coordinate actions with
nearby robots. Each agent then uses these position estimates to find a
safe-reachable set which is characterized by a generalized Voronoi partition.
Our algorithm computes a projection onto this set, which we show reduces to an
efficiently solvable convex optimization problem. Our method is amenable to fast
convex optimization solvers, resulting in computations times of approximately 18
ms for 100 obstacles in 3D, including setup and solution time. Furthermore, we
prove that if each agent uses this policy then mutual collision avoidance is
guaranteed.

This paper is organized as follows. The remainder of this section discusses
related work. Section~\ref{probForm} formulates the mutual avoidance problem and
gives the necessary mathematical background on generalized Voronoi partitioning.
Section~\ref{avoid} describes the collision-avoidance algorithm and provides a
collision avoidance guarantee.  Section~\ref{cvxProj} formalizes the projection
problem and describes the resulting convex optimization for ellipsoidal
uncertainties in detail. Finally, section~\ref{results} shows our method's
computational performance. Additionally, we show in 2D and 3D simulations that
all agents avoid each other and navigate to their goal locations despite their
positional uncertainty of other agents.

\subsection{Related Work}
The most closely related methods for fully distributed collision
avoidance in the literature are the velocity obstacle (VO) methods,
which can be used for a variety of collision avoidance strategies.
These methods work by extrapolating the next position of an obstacle
using its current velocity. One of the most common tools used for
mutual collision avoidance is the Reciprocal Velocity Obstacles
(RVO)~\cite{rvo:2008,rvo_n:2011,ocra:2016} method in which each agent
solves a linear program to find its next step. The Buffered Voronoi
Cell (BVC)~\cite{bvc:2017} method provides similar avoidance
guarantees but does not require the ego agent to know other agents'
velocities, which can be difficult to estimate accurately. The BVC
algorithm opts instead for defining a given distance margin to compute
safe paths. BVC methods have been coupled with other decentralized
path planing tools~\cite{heirRobust:2019} in order to successfully
navigate more cluttered environments, but require that the other
agents' positions are known exactly.

While both VO and BVC methods scale very well to many (more than 100)
agents, they also require perfect state information of other agents'
positions (BVC), or positions and velocities (RVO). In many practical
cases, high accuracy state information, especially velocity, may not
be accessible as agents are estimating the position of the same
objects they are trying to avoid. Extensions to VO that account for
uncertainty have been studied under bounded~\cite{boundedLoc:2012} and
unbounded~\cite{PRVO:2017} localization uncertainties by utilizing
chance constraints. While these have been extended to decentralized
methods~\cite{chanceCon:2019}, they assume constant velocity of the
obstacles at plan time. Combined Voronoi partitioning and estimation
methods have been studied for mult-agent path planing
tasks~\cite{transport:2014}, but still require communication to build
an estimate via consensus.  In contrast, our method does not require
any communication or velocity state information, nor does it require
the true position of the other agents. Instead, the algorithm uses
only an estimate of the current position of the nearby agents and
their reachable set within some time horizon.

Our algorithm takes a nominal desired trajectory or goal point (which
can come from any source, much like~\cite{rvo:2008,bvc:2017}), and
returns a safe next step for an agent to take while accounting for
both the uncertainty and the physical extent of the other agents in
the vicinity.  The focus of this work is on fast, on-board refinement
rather than total path planning. More specifically, while the
algorithm presented could be used to reach a far away goal point, it
is likely more useful as a small-to-medium scale planner for reaching
waypoints in a larger, centrally or locally planned, trajectory.

Similarly, single agent path planners such as A* or
rapidly-exploring random trees (RRT) can be applied to the multi-agent
case, but the solution times grow rapidly due to the exploding size of
the joint state space. For example, graph-search methods can be
partially decoupled~\cite{MSTAR:2011} to better scale for larger
multi-agent systems, but can explore the entire joint state-space in
the worst case. Fast Marching Tree (FMT) methods~\cite{fmt::2015} are
similar to RRTs in that they dynamically build a graph via
sampling. But, while FMT methods have better performance in higher
dimensional systems, they still require the paths to be centrally
calculated.  Using barrier functions~\cite{barrier:2016}, A* can also
be used in dense environments for decentralized multi-agent planning,
but require the true position of all the other agents. On the other
hand, while all of these methods require global knowledge and large
searches over a discrete set, they can be used as waypoint
generators that feed into our method---for use, \eg, in cluttered
environments.

Optimization methods that use sequential convex programming
(SCP)~\cite{scp:2013,scpQuad:2012,scpSpace:2014} have also been studied for
multi-agent path planning; however, these algorithms are still centralized and
may exhibit slow convergence, making them unreliable for on-line planning. For
some systems, these methods can be partially decoupled~\cite{iscp:2015},
reducing computation time at the cost of potentially returning infeasible paths.
Our method, in comparison, is fully decentralized and produces an efficient convex
program for each agent. The solution of this program is a safe waypoint for the
agent, which, unlike SCP methods, requires no further refinment.


\section{Problem Formulation}\label{probForm}
Consider a group of $N$ dynamic agents. Let $x_i(t) \in \reals^n$ be the
position of agent $i$ (also referred to as the ``ego agent") for $i=\ito{N}$ at 
time $t=\ito{T}$, where each agent
satisfies single integrator dynamics,
\begin{align}
    x_i(t+1) = x_i(t) + u_i(t), ~~ i=\ito{N}, ~~ t=\ito{T},
\label{dynamics}
\end{align}
with control $u_i(t) \in \reals^n$ and $\norm{u_i(t)} \leq \umax$. In addition,
every agent $i$ will maintain a set-based estimate of the location of every
other agent $j \ne i$ as a set $\ellipse^j_i(t)$, such that
\begin{equation}
    x_j(t) \in \ellipse^j_i(t), ~~ i=\ito{N}, ~~ t=\ito{T}.
\label{obvs}
\end{equation}
In other words, the true position of agent $j$ at time $t$ must always be inside
the noisy estimate $\ellipse^j_i(t)$ maintained by agent $i$. In practice, this
set can be obtained from a high probability confidence ellipsoid of a Bayesian
filter or from a set-membership filter~\cite{bertsekas::1971}. We do not
restrict the size of this uncertainty region, though we note that large
uncertainties will restrict an agent's possible actions.

Let the \emph{safe-reachable} set $S_i(t)$ be the generalized Voronoi
cell~\cite{genvoro::1981,huang:2012,shah:2019} generated between agent $i$'s 
current position, $x_i(t)$, and the set of estimates, $\ellipse^j_i$, that agent
$i$ has of every agent $j$ at time $t$. That is,
\begin{equation}\label{genVoro}
    S_i(t) = \left\{ z ~\middle|~ \norm{z-x_i(t)} \leq \|z-\ellipse^j_i(t)\|_2 ~~
    \forall j=\ito{N}, ~  j\neq i \right\}
\end{equation}
where $\|z-\ellipse^j_i(t)\|_2$ is the distance-to-set metric
$\|z-\ellipse^j_i(t)\|_2=\inf_{y\in \ellipse^j_i(t)} \norm{z-y}$.

More explicitly, $S_i(t)$ is the set of points which lie closer to
agent $i$ than to any possible position of agent $j$ within its
uncertainty set $\ellipse^j_i$. This has a natural interpretation in
the case where all agents share the same single-integrator dynamics
(though this condition is not necessary for the algorithm, in
general): $S_i(t)$ is the set of all points which agent $i$ is
guaranteed to reach from its current position before any other
agent. We emphasize that since each agent will often maintain
different estimates of the positions of other agents, the sets
$S_i(t)$ generally do not form a partition of $\reals^n$, as the union
$\cup_i S_i(t)$ does not equal the entire
space. Figure~\ref{projCurves} shows the generalized Voronoi cell
boundaries in an example simulation: each agent is shown with the
estimates it has of other agents, along with the agent's goal position
and projected goal position.
\begin{figure}
\centering
\includegraphics[width=.6\textwidth]{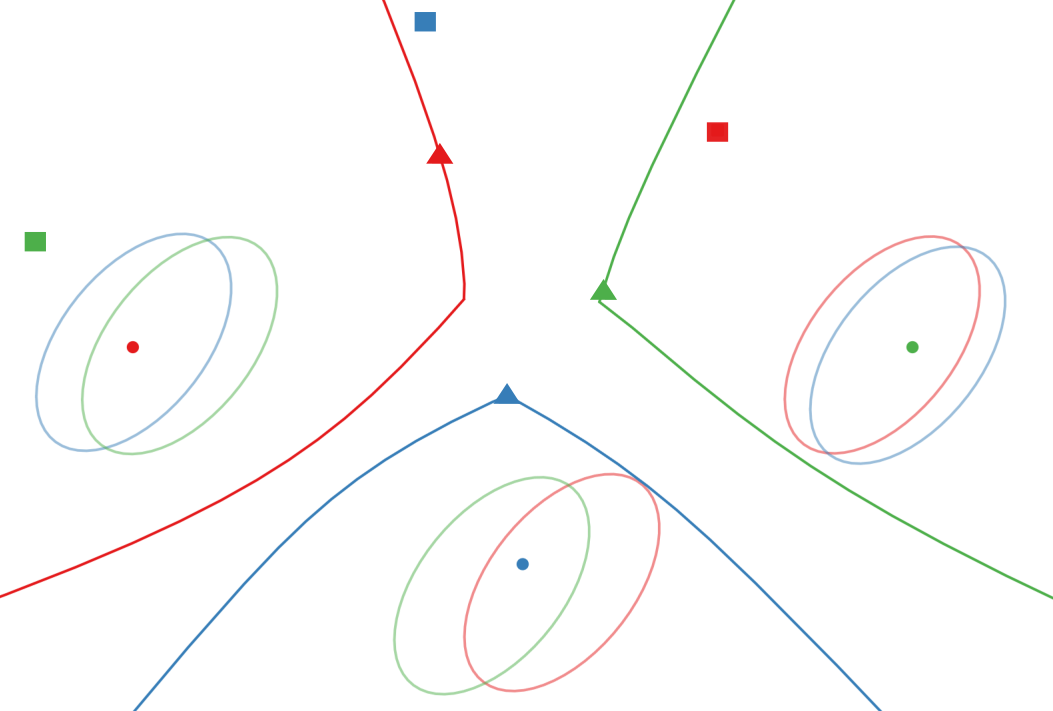} 
\caption{Generalized Voronoi cells of each
          (blue, red, green) agent's position (circle) with
          goal point (square) and safe projected goal point
          (triangle). Each agent uses an ellipsoidal estimate
          (colored ellipses) of each other agent's position to
          generate its own cell.}
\label{projCurves}
\end{figure}
\begin{figure}
\centering
\includegraphics[width=.5\textwidth]{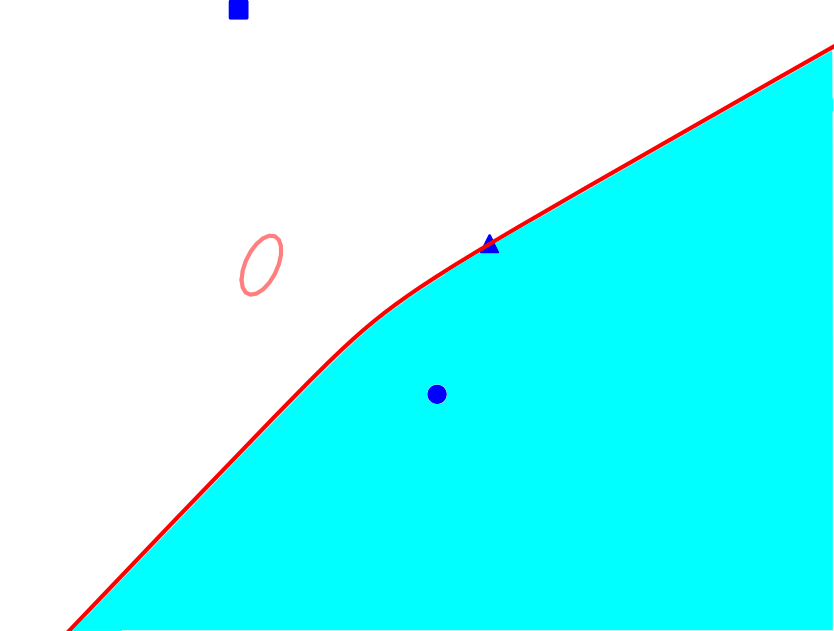}
\captionof{figure}{Projection (blue triangle) of a goal point
  (blue square) onto a generalized Voronoi region (cyan
  region). The boundary (red curve) is equidistant from the
  agent's position (blue point) and the closest point inside
  the ellipsoidal estimate of the obstacle (red ellipse).}
\label{fig::proj}
\end{figure}
We also note that $S_i(t)$ is a relatively conservative 
estimate of the set of collision-free points. Let
\begin{equation*}
	\ellipse_i(t) = \bigcup_{\substack{j=1\\ j\ne i}}^N \ellipse^j_i(t),
\end{equation*}
that is, $\ellipse_i(t)$ is the set of all points which could potentially
include another agent---from the perspective of agent $i$---then the
safe-reachable set satisfies
\begin{equation*}\label{safe-set}
    S_i(t) \subseteq \reals^n \setminus \ellipse_i(t).
\end{equation*}
Additionally, we note that each set $S_i(t)$ is the intersection of an infinite
number of half-spaces since~\eqref{genVoro} can be rewritten as
\begin{equation}\label{vc_cvx}
    S_i(t) = \bigcap_{y \in \ellipse_i(t)}\left\{ z ~ \middle| ~ \norm{z-x_i}
    \leq \norm{z-y}\right\},
\end{equation}
which implies that $S_i(t)$ is a convex set~\cite[\S2.3.1]{cvxbook}. This means
that a projection of an arbitrary point (such as a goal destination or
way-point) onto this set can be formulated as a (potentially infinitely large)
convex optimization problem. In the next section, we describe an algorithm which
uses these projections and then prove collision-free guarantees for this
algorithm, while in the following section, we show how these projections can be
efficiently computed.

Normally, uncertainties are defined for points in space (\ie, the center of a
robot in $\reals^3$); however, without physical extent, collision is a measure
zero event. To guarantee collision avoidance, we must account for the
uncertainty as well as the physical size of the other agents.  Assuming that the
robot's physical extent can be represented by an ellipsoid, we can combine the
physical size ellipsoid with the uncertainty ellipsoid such that their Minkowski
sum is contained inside another bounding ellipsoid $\bar{\ellipse}$.
In particular, we seek a $\bar{\ellipse} \supseteq\ellipse_\mathrm{estimate}
\oplus \ellipse_\mathrm{size}$. We can find a small ellipsoid which satisfies
this propety efficiently (in fact, analytically) by solving a minimum trace
optimization problem, the details can be found in~\cite{liu::2016}.
Additionally, using ellipsoidal margins allows us to potentially account for
more complicated effects that could not be reasonably represented by spherical
margins. For example,~\cite{downwash:2017} uses ellipsoids elongated in the
$z$-axis to generate keepout zones above and below quadrotors which account for
the downwash effect of the propellers.

\section{Collision Avoidance}\label{avoid}
We present an algorithm which uses the projections described in
section~\ref{probForm} to reach a given goal. We then present a general proof of
collision avoidance guarantees for a class of algorithms including the
one presented in~\cite{bvc:2017} and Algorithm~\ref{alg:gen} (below), along 
with its natural extensions.

\subsection{General Algorithm}
To simplify notation, let $R_i(t)$ be the reachable set of positions of each
agent $i$ at time $t$. We note that, in our case, $R_i(t)$ is the intersection
of the closed $\ell_2$-ball with radius $\umax$ and the generalized Voronoi cell
of agent $i$ at time $t$.  We also note that this is a compact set which always
contains $0$ if it is not empty.

Let $\ellipse_i^j(t)$ be the uncertainty ellipsoid with nonzero volume that
agent $i$ has of agent $j \ne i$ at time $t$ and let $P_i(x_i, \ellipse_i, t)$
be the projection of agent $i$'s goal into the intersection of its generalized
Voronoi cell and its reachable set at time $t$. As a technical requirement, we
will require that agent $j$'s true position lies in the interior of
$\ellipse_i^j(t)$. Since this projection can fail (\eg, when agent $i$'s
uncertainty of agent $j$'s position is large enough to include the position of
agent $i$), we allow the projection function to return either a point in
$\reals^n$, or a symbol, $\neg$, which indicates that the projection
failed---\ie, there is no safe point to move to.

\begin{algorithm}[H]
\caption{General projection algorithm}
\label{alg:gen}
\begin{algorithmic}
\STATE $i \gets$ current agent.
\STATE $x_i(1) \gets$ initial position of agent $i$.
\FOR{$t=1, \dots, T$}
    \STATE $q \gets P_i(x_i, \ellipse_i, t)$
	\IF{$q = \neg$}
		\STATE{$x_i(t+1) \gets x_i(t)$}
	\ELSIF{$q \ne \neg$}
		\STATE{$x_i(t+1) \gets q$}
	\ENDIF
\ENDFOR
\end{algorithmic}
\end{algorithm}

Algorithm~\ref{alg:gen} states that, if the projection fails,
the agent does not move from its current position; otherwise, the agent moves 
towards the next projection.

\begin{theorem}\label{thm:alg1}
    Algorithm~\ref{alg:gen} is collision free, assuming all agents start from a
    collision free configuration.
\end{theorem}


\begin{proof}
First, we note the following properties of the projection function defined
above. If the function yields a point in $\reals^n$ (and not $\neg$) then the
following statements are satisfied,
\begin{equation}\label{eq:proj}
\begin{aligned}
    P_i(x_i, \ellipse_i, t) &\in x_i(t) + R_i(t), ~~~ t=1, \dots, T\\
    P_i(x_i, \ellipse_i, t) &\not\in x_j(t) + R_j(t), ~~~ t=1, \dots, T, ~~ j=1,\dots, N, ~~ j\ne i,
\end{aligned}
\end{equation}
where the sum of a vector $v \in \reals^n$ and a set $M \subseteq \reals^n$ is
defined as $v+M=\{v + w  \mid w \in M\}$. In other words,
condition~\eqref{eq:proj} requires that the projection, if it returns a point,
must (a) return a point which is in the agent's reachable set, and (b) cannot
return a point which could lie in region of another agent's reachable set, since
agent $i$ does not know where agent $j$ is headed.

The first statement (a) comes from the fact that, by definition, the projection
function gives a point in the intersection of the reachable set and the
generalized Voronoi cell. The second (b) comes from the fact that the set which
we project into contains only the points which are reachable by agent $i$ before
they are by agent $j$, and agent $j$ maintains an uncertainty ellipsoid of agent
$i$'s position which includes agent $i$'s true position. This implies that the
reachable set of agent $j$ at time $t$ cannot include the projection of agent
$i$ at time $t$.

Now, assume that all agents start a positive distance apart at time $t=1$ and
consider any two agents $i\ne j$. Recall that our projection function $P_i(x_i,
\ellipse_i, t)$ returns a point or returns~$\neg$. Suppose that a point is
returned, then, by definition of algorithm~\ref{alg:gen} and by~\eqref{eq:proj},
we know that
\[
x_i(t+1) = P_i(x_i, \ellipse_i, t) \not \in x_j(t) + R_j,
\]
so $x_i(t+1)$ must be a positive distance from $x_j(t) + R_j(t)$, by compactness of
$R_j(t)$. Since, by definition, $x_j(t) + R_j(t)$ contains the point $x_j(t+1)$, then
$x_i(t+1)$ is always a positive distance away from $x_j(t+1)$.

Now, assume that $P_i(x_i, \ellipse_i, t)$ returns $\neg$, then, for the other
agent, $P_j(x_j, \ellipse_j, t)$ either returns $\neg$ (in which case no
collision happens, since neither agent has moved, by definition of
algorithm~\ref{alg:gen}) or $P_j$ returns a point which satisfies
\[
    x_j(t+1) = P_j(x_j, \ellipse_j, t) \not \in x_i(t) + R_i(t).
\]
But, since agent $i$ has not moved, $x_i(t+1) = x_i(t) \in x_i(t) + R_i(t)$ (since
$0 \in R_i(t)$, by definition), so $x_j(t+1)$ is a positive distance away from
$x_i(t+1)$, again by compactness of $R_i(t)$.

Since this is true for any two agents, then all agents stay a positive distance
apart from each other for each time $t=1, \dots, T$, and no collision happens.
\end{proof}

\subsection{Generalizations and discussion}
While the proof above is presented only in the context of
algorithm~\ref{alg:gen} with the projection function and reachable sets
specified in \S\ref{probForm}, the proof is almost immediately generalizable to
many other projection functions
and reachable sets. We present a few of these
generalizations below.

\paragraph{Proof requirements.} In the proof of theorem~\ref{thm:alg1}, we only
used the following three facts: (a) the projection function
satisfies~\eqref{eq:proj}, (b) that $R_i(t)$ was compact for each agent $i$ at
each time $t$, and (c) that 0 is in the reachable set ($0 \in R_i(t)$). This
means that any set of agents and projection functions that satisfy the above
conditions are immediately guaranteed to be collision free, if they use
algorithm~\ref{alg:gen}. There are many such projection functions, for example
the trivial function (which always returns $\neg$) and functions which are
potentially very complicated and depend on the histories of the uncertainties,
but all cases guarantee collision avoidance so long as these conditions are
satisfied.

\paragraph{Minimum distance.} The minimum distance between agents depends on the
aggressiveness of the projection function---in the sense that the proof above only
guarantees a non-zero (but arbitrarily small) separation. We can
give slightly better bounds by under-projecting to ensure that, for every $i\ne
j$ and time $t$,
\begin{equation}\label{eq:min-eps}
    \|P_i(x_i, \ellipse_i, t)- (x_j(t) + R_j(t))\|_2 \ge \eps > 0.
\end{equation}
This guarantees that every pair of agents will have a separation of at least
$\eps$. 

It is sometimes easy to generate the under-projection condition
in~(\ref{eq:min-eps}). For example, this is possible in the case where $R_j(t)$
is a convex set\footnote{More generally, a star-shaped domain around $0 \in
R_j(t)$ would suffice.} and $\{P_j\}$ is any set of valid projection functions.
To generate this, we choose the $q$ of algorithm~\ref{alg:gen} to be a convex
combination of the current position and the projected position which satisfies
the inequality above. This ensures that every pair of agents is at least
$\eps$-separated---assuming all agents start with at least $\eps$ separation at
time $t=1$---for all $t=1, \dots, T$.

\paragraph{Uncertainties.} The sets $\ellipse_i^j(t)$ for $t=1, \dots, T$ and
$j=1, \dots, N$ with $j\ne i$ only play a role as the arguments to $P_i$, but
there is no requirement that these measurements be accurate or even bounded in
any sense. It is possible for one agent to temporarily have large uncertainty
about the positions of all other agents (\eg, in the case of GPS loss) before
gaining a more accurate measurement and continuing to its objective. Of course,
the usefulness of the projection function will heavily depend on the quality of
these projections, but we are guaranteed to be collision free at every point in
time, independent of these assumptions, so long as condition~(\ref{eq:proj}) is
satisfied.

\paragraph{Relaxing the reachability condition.} In the proof of
theorem~\ref{thm:alg1} we assume that the current position is always a reachable
state for agent $i$ (\ie, that $0 \in R_i(t)$). While this is the case in
single-integrator dynamics, for example, it is not the case in general. It is
possible to weaken this assumption slightly by ensuring that the agent can stop
within some $\eps$-ball, such that some point in $x_i(t) + \eps B$ is always
reachable (where $\eps > 0$ and $B = \{x \in \reals^n \mid \norm{x} \le 1\}$)
from state $x_i(t)$, but the assumptions on the projection functions, $P_i$,
must be strengthened considerably from the general ones given in
condition~\eqref{eq:proj}.

\paragraph{Asymmetric dynamics.} It is rarely the case that the dynamics of the
agents in question are single-integrator dynamics with the same maximal input
for all agents, as we assumed at the beginning of this proof. In the case that
the dynamics are asymmetric among agents, we can guarantee
condition~\eqref{eq:proj} with the projections presented in this section by
simply expanding the uncertainty ellipsoid by a margin which includes the
reachability of the other agents. That is, we can replace the uncertainty
ellipsoid that agent $i$ has of agent $j$, originally given by $\ellipse_i^j(t)$
at time $t$ with $\ellipse_i^j(t) \oplus R_j(t)$ (or any other outer envelope 
of this
set). Using this new uncertainty to generate a safe Voronoi region and using its
corresponding projection, as given in \S\ref{safe-set}, is then guaranteed to be
collision free.

\section{Projecting onto Generalized Voronoi Cells}\label{cvxProj}
In this section, we describe a solution to the problem of efficiently projecting
a point into a generalized Voronoi region.

First, we construct a program which is equivalent to finding a projection into a
convex set of the form~\eqref{vc_cvx}, but may not be easy to solve as its
constraint is not representable in any standard form. We then use Lagrange
duality to construct a convex constraint that is at least as strict as the
original and use strong duality to show that this constraint is
equivalent to the original problem. Finally, we provide a conic problem for
the case of ellipsoid generated Voronoi regions with the constraint explicitly
parametrized by the ellipsoid parameters $(\mu, \Sigma)$.  We also show that,
for the ellipsoidal case, the constraint is represented by a sum of
quadratic-over-linear terms, implying that the convex program is a second order
cone program (SOCP) and can therefore be solved quickly by embedded solvers.

\subsection{Problem Statement}
As in~\cite{bvc:2017}, in order to execute the collision-avoidance strategy,
each agent must project its goal point onto its safe-reachable set as defined
in~\eqref{vc_cvx}. This problem is always convex since the Voronoi region is
generated by an (arbitrary) intersection of hyperplanes~\cite[\S2.3.1]{cvxbook},
which always results in a convex set.

Consider the case in which the Voronoi cell is generated by a point $x$ and
a single convex set, $\ellipse$, defined by
\[
\ellipse = \{y~|~f(y) \le 0\},
\]
where $f:\reals^n \to \reals^m$ is a convex function and the inequality is taken
elementwise. This allows us to write the problem as finding a projection point $x$ which
solves
\begin{equation}
    \label{opt:proj}
	\begin{aligned}
		& \underset{x}{\text{minimize}}
		& & \normsq{x-\xgoal} \\
		& \text{subject to}
        & & x \in V(\xcurr, \ellipse)
	\end{aligned}
\end{equation}
where $\xcurr$ is the current position of the agent, $\xgoal$ is the goal, and
$V(\xcurr, \ellipse)$ is the agent's generalized Voronoi region,
\[
V(\xcurr, \ellipse)= \{z \mid \norm{z-\xcurr}\leq \inf_{y\in \ellipse}\norm{z-y}
\}.
\]
For some special cases of $\ellipse$, such as a circle or sphere, the
constraint's infimum can be found analytically. However, this may not be
possible for arbitrary sets.  For example, in the case where $\ellipse$ is an
ellipsoid, no analytical distance function has been found~\cite{p2eDist:2018}.

Using the definitions of $\ellipse$ and $V(\xcurr, \ellipse)$, we can 
write~\eqref{opt:proj} as
\begin{equation}
    \label{eq:main}
	\begin{aligned}
		& \underset{x}{\text{minimize}}
		& & \normsq{x - \xgoal} \\
		& \text{subject to}
        & & \normsq{\xcurr} - 2x^T\xcurr \le \inf_{f(y) \le 0}\left(\normsq{y} -
        2x^Ty\right).
	\end{aligned}
\end{equation}
This is the problem formulation we will use throughout. Figure~\ref{fig::proj}
gives an illustration of the projection problem.

\subsection{Lagrange Duality}
At the moment, it is not obvious how to represent the the constraint given in
problem~\eqref{eq:main} in a standard or easy-to-solve form.

If we can find a lower bound to the right hand side of the constraint in
problem~\eqref{eq:main} which makes the resulting problem easy to solve, then
we can find a feasible (\ie, safe) point, $x$, which may not be optimal. More
concretely, suppose we have a function $\hat g$ which satisfies, for every $x$,
\[
    \hat g(x) \le \inf_{f(y) \le 0} \left(\normsq{y} - 2x^Ty\right),
\]
then any $x$ which satisfies
\[
\normsq{\xcurr} - 2x^T\xcurr \le \hat g(x),
\]
also satisfies
\[
\normsq{\xcurr} - 2x^T\xcurr \le \inf_{f(y) \le 0} \left(\normsq{y} - 2x^Ty\right),
\]
making $x$ a feasible point for problem~(\ref{eq:main}). One standard way of
forming such a lower bound is via Lagrange duality~\cite[\S5.1]{cvxbook}. The 
Lagrangian of the infimum in~\eqref{eq:main}
\[
\lagr(y, x, \lambda) = \normsq{y} - 2x^Ty + \lambda^Tf(y),
\]
with $\lambda \in \reals^m$ and $\lambda \ge 0$. This gives us the Lagrange dual
function
\begin{equation}\label{larg}
g(x, \lambda) = \inf_y \lagr(y, x, \lambda) = \inf_y \left(\normsq{y} - 2x^Ty + \lambda^Tf(y)\right).
\end{equation}
By weak duality~\cite[\S5.1.3]{cvxbook}, for every $\lambda \ge 0$ and each $x$,
\[
g(x, \lambda) \le \inf_{f(y) \le 0} \left(\normsq{y} - 2x^Ty\right),
\]
as required.

Additionally, since the dual function $g$ is jointly concave in
both $\lambda$ and $x$, the resulting inequality constraint is convex,
\[
    \normsq{\xcurr} - 2x^T\xcurr \le g(x, \lambda).
\]

\paragraph{Strong duality.}
Due to the lower bound property of $g$, the optimization problem,
\begin{equation}
	\label{eq:dual-form}
	\begin{aligned}
		& \underset{x,\ \lambda}{\text{minimize}}
		& & \normsq{x-\xgoal} \\
		& \text{subject to}
		& & \normsq{\xcurr} - 2x^T\xcurr \le g(x, \lambda) \\
        &&& \lambda \ge 0,
	\end{aligned}
\end{equation}
is potentially more restrictive than the original and is thus an upper bound
on the optimal objective of problem~\eqref{eq:main}. Due to strong duality
holding in almost all cases of interest (\ie, all cases where the set $\ellipse$
has nonempty interior), we will see that problems~\eqref{eq:dual-form} and
\eqref{eq:main} are equivalent which implies that every optimal solution to
problem~\eqref{eq:dual-form} is an optimal solution to problem~\eqref{eq:main};
in other words, the relaxation provided is tight.

In particular, Slater's condition holds for any convex set $\ellipse$ with
non-empty interior (in three dimensions, this would be any convex set with
nonzero volume). Since Slater's condition implies strong
duality~\cite[\S5.3.2]{cvxbook}, then for each $x$, there exists some
$\lambda^\star \ge 0$ such that \[ g(x, \lambda^\star) = \inf_{f(y) \le 0}
\left(\normsq{y} - 2x^Ty\right), \] which means that a solution to
problem~(\ref{eq:dual-form}) is \emph{always} a solution to
problem~\eqref{eq:main}.

Given an arbitrary convex function $f$, it is unclear if the
associated function $g$ defined by~\eqref{larg} has an analytic form or is even
easy to evaluate. In the following section, we derive an analytic form for $g$
in the case that $\ellipse$ is an ellipsoid. In the appendix, we also derive an
analytic form for $g$ for polyhedral sets. We also show that it is possible to
construct more complicated sets from the union and intersection of these
ellipsoidal and polyhedral `atoms' allowing the user to specify complicated, 
non-convex sets as the uncertainty regions of other agents or
obstacles.

\subsection{Constraints for Regions Generated by Ellipsoids}
If the set $\ellipse$ is specified by a convex quadratic constraint, as in the case of
ellipsoids, then the set $\ellipse$ can be written as,
\begin{align*}
	\ellipse
	&=\{y~|~f(y) \le 0 \}
	=\{y~|~(y - \mu)^T\Sigma^{-1}(y - \mu) \le 1\}	
\end{align*}
with $\mu \in \reals^n$ and $\Sigma \in \symm^{n}_{++}$, the positive definite
matrix cone, representing the center and shape of the uncertainty, respectively.
Note that the minimum of the convex quadratic,
\[
    y^TAy - 2b^Ty,
\]
with $A$ positive definite is given by
\[
\inf_{y}\left(y^TAy - 2b^Ty\right) = -b^TA^{-1}b.
\]
Here, $y^* = A^{-1}b$, is found by setting the gradient to zero---this is
necessary and sufficient by convexity and differentiability. The dual function
is then,
\begin{align*}
g(x, \lambda)
&= \inf_y\left(\normsq{y} - 2x^Ty + \lambda((y - \mu)^T\Sigma^{-1}(y - \mu)- 1)\right)\\
&= \inf_y\left(y^T(\lambda\Sigma^{-1} + I)y - 2(x +\lambda\Sigma^{-1}\mu)^Ty + \lambda\left(\mu^T\Sigma^{-1}\mu - 1\right)\right)\\
&=-(x + \lambda\Sigma^{-1}\mu)^T(\lambda\Sigma^{-1} +  I)^{-1}(x +
\lambda\Sigma^{-1}\mu) + \lambda(\mu^T\Sigma^{-1}\mu - 1).
\end{align*}
Though the function $g$ can immediately be written in standard semidefinite
program (SDP) form via the Schur complement, it is possible to convert it into a
second-order cone constraint form, which is usually more amenable to embedded
solvers (\eg, ECOS~\cite{ecos:2013}).

First, since $\Sigma$ is positive definite, it has an eigendecomposition
$\Sigma = UDU^T$, where $U \in \reals^{n\times n}$ is an orthogonal matrix that
satisfies $UU^T = U^TU = I$ and $D \in \reals^{n\times n}$ is a diagonal matrix
with positive entries. Using the fact that $\Sigma^{-1} = UD^{-1}U^T$, we can
write
\[
(\lambda \Sigma^{-1} + I)^{-1}
= (\lambda  UD^{-1}U^T +  UU^T)^{-1}
= U(\lambda D^{-1} + I)^{-1}U^T,
\]
which gives,
\begin{equation}\label{dual}
	\begin{aligned}
		g(x, \lambda)
		&= -(U^T(x + \lambda\tilde\mu))^T(D^{-1} + \lambda I)^{-1}(U^T(x +
        \lambda\tilde\mu))
	+ \lambda(\tilde\mu^T\Sigma\tilde\mu - 1)\\
		&= -\sum_{i=1}^n \frac{(u_i^Tx + \lambda u_i^T\tilde \mu)^2}{D_{ii}^{-1}
        + \lambda} +\lambda(\tilde\mu^T\Sigma\tilde\mu - 1),
	\end{aligned}
\end{equation}
where $\tilde \mu = \Sigma^{-1}\mu$ and $u_i$ is the $i$th column of $U$. After 
substituting~\eqref{dual} into~\eqref{eq:main} we obtain the following 
optimization problem,	
\begin{equation}
    \label{eq:final}
	\begin{aligned}
		& \underset{x, \ \lambda}{\text{minimize}}
		& & \normsq{x-\xgoal} \\
		& \text{subject to}
		& & \normsq{\xcurr} - 2x^T\xcurr +\sum_{i=1}^n \frac{(u_i^Tx + \lambda u_i^T\tilde \mu)^2}{D_{ii}^{-1}
        + \lambda} \le \lambda(\tilde\mu^T\Sigma\tilde\mu - 1) \\
    	& & &  \lambda \ge 0,
	\end{aligned}
\end{equation}
where $\Sigma$ and $\tilde \mu = \Sigma^{-1}\mu$ are given by the
uncertainty of the other agent's position, and $\xgoal$ and $\xcurr$
are known by the ego agent (\ie, the agent solving the optimization problem).

As there is a standard approach for converting the sum of $n$
quadratic-over-linear terms into $n$ second-order cone (SOC) constraints (\eg,
see~\cite{lobo:1998}) and an affine constraint, this formulation can then be
used directly with embedded SOCP solvers. It can also be automatically converted
to an SOCP by modeling languages such as CVXPY~\cite{cvxpy:2018}.

\section{Simulation Results}\label{results}
\subsection{Projection Implementation}\label{sec:results:impl}
To get an accurate estimate of the speed of the projection algorithm,
the optimization problem outlined in~\eqref{eq:final} was implemented in the
Julia language~\cite{bezanson:2017} using the JuMP mathematical programming
language~\cite{dunning:2017} and solved using ECOS~\cite{ecos:2013}.
285 instances of the problem were generated with 100 randomly generated
ellipsoids in $\reals^3$. Timing and performance results for generating and
solving the corresponding convex program can be found in
table~\ref{tab:table_all}. Figure~\ref{fig:timings} shows how the performance
scales as the number of other agents increases. All times reported are on a
2.9GHz 2015 dual-core MacBook Pro.

\begin{center}
    \begin{minipage}[]{0.47\textwidth}
        \centering
        \includegraphics[width=0.99\textwidth]{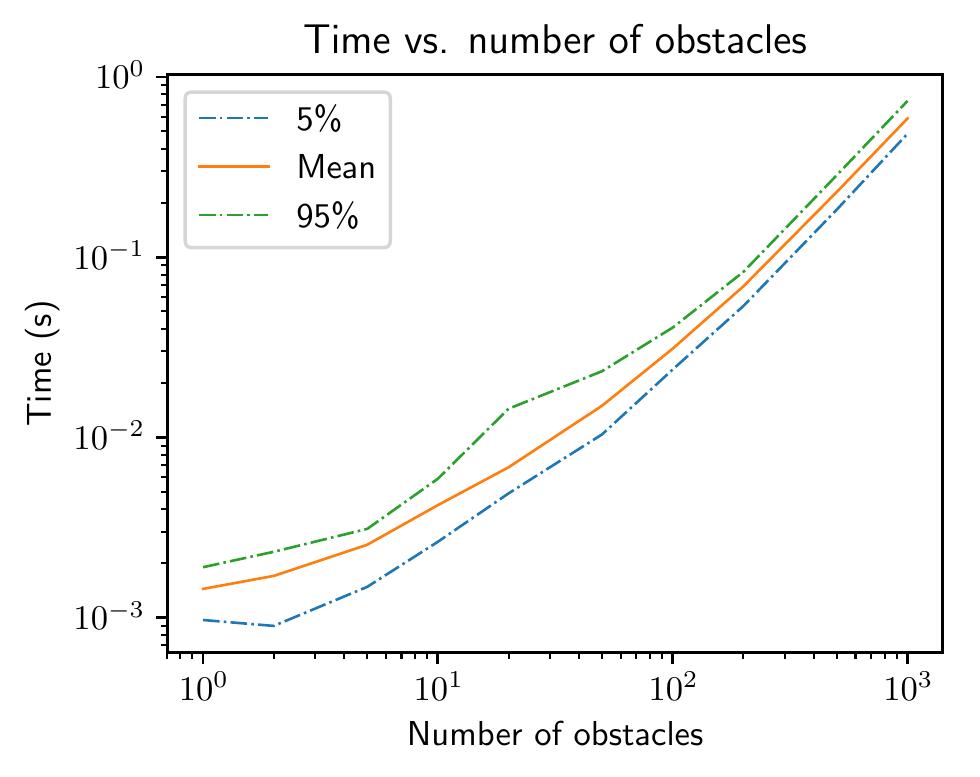}
        \captionof{figure}{Graph showing total time for generating and solving
        optimization problem~\eqref{eq:final} as a function of the number of
        ellipsoids in the problem when solved using the ECOS solver. Note the
        logarithmic scales on both axes.}
        \label{fig:timings}
    \end{minipage}
    \hspace{1mm}
    \begin{minipage}[]{0.47\textwidth}
        \centering
        \begin{tabular}{@{}lr@{}}
            \toprule
            \textbf{Time} & \textbf{Total (GC \%)} \\
            \midrule
            Minimum & 13.20\si{\milli\second} (00.00\%)\\
            Median & 17.12\si{\milli\second} (00.00\%) \\
            Mean  & 17.55\si{\milli\second} (08.80\%) \\
            Maximum  & 36.77\si{\milli\second} (10.59\%) \\
            \bottomrule
        \end{tabular}
        \vspace{1em}
        \captionof{table}{Table reporting times with garbage collection (GC) 
        precentage for building and solving problems with 100 randomly 
        generated 3D ellipsoids. Statistics are based on
        285 instances and were obtained from the
        \texttt{BenchmarkTools.jl}~\cite{chen:2016} package.}
        \label{tab:table_all}
    \end{minipage}
\end{center}

\subsection{Trajectory Simulations}
The projection algorithm was implemented in both 2D and 3D with a varying number
of agents.\footnote{A video of the simulations can be found at
\url{https://youtu.be/oz-bMovG4ow}.} Each agent knows their position exactly and
maintains a noisy estimate of other agents' positions, with uncertainties
represented as ellipsoids. This estimate is updated by a set-membership based
filter~\cite{bertsekas::1971,liu::2016,shah:2019}, a variant of the Kalman
filter. We expand the uncertainty ellipsoid by a given margin to account for the
robot's physical size.  If this margin is also ellipsoidal then a small
ellipsoid which contains the Minkowski sum of the uncertainty ellipsoid and the
margin can be found in closed form~\cite{liu::2016,shah:2019}. This new bounding
ellipsoid is used in the projection algorithm to account for a user defined
margin, along with the uncertainty ellipsoid containing the noisy sensor
information.

Figure~\ref{fig::2dsim-dist} shows the minimum inter-agent distances for each 
agent in the simulation scenario mentioned above. The collision threshold was 
set to $.4\si{\meter}$, twice the radius of the agents. Although our method 
results in longer paths, it remains collision free, while RVO's paths result in 
collision. 
\begin{figure}	
	\centering
    \includegraphics[width=.99\textwidth]{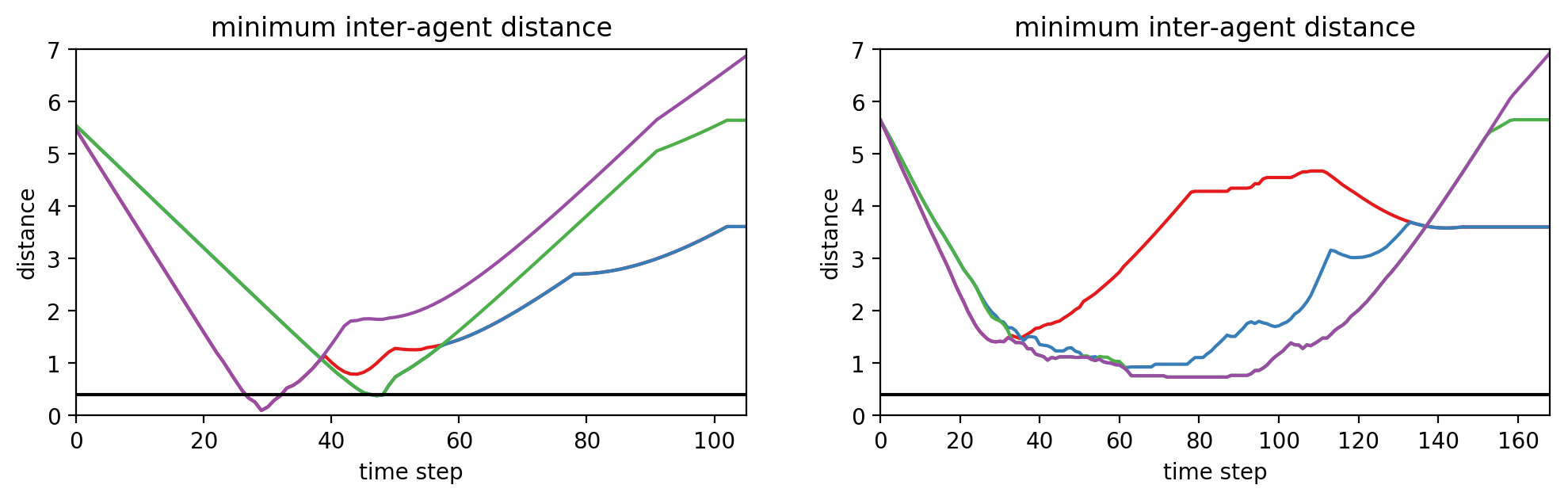}
	\caption{Inter-agent distances RVO (left) and our method (right). The RVO 
	simulation results in inter-agent distances below the collision threshold 
	(black line).}
	\label{fig::2dsim-dist}
\end{figure}

Figure~\ref{fig::3dsim} shows six instances of a 3D simulation with 10
agents. The agents start at the sides of a
$10\si{\meter}\times10\si{\meter}\times10\si{\meter}$ cube and are constrained
to a maximum speed of $6\si{\meter/\second}$ and a maximum measurement error set
to $1.0\si{\meter}$.  The algorithm was run at $60\si{\hertz}$.
The agents, displayed as quadrotors, each have a bounding box of
$0.45\si{\meter}\times0.45\si{\meter}\times0.2\si{\meter}$ and an additional
ellipsoidal margin with axis lengths of $.3\si{\meter}$ in the $x$ and
$y$ dimensions, and $1.2\si{\meter}$ in the $z$ dimension.
This margin effectively gives a buffer of $0.75\si{\meter}$ in the $xy$ plane
and a large buffer of $1.3\si{\meter}$ in $z$. We assume
non-spherical margins in this simulation since, in the case of quadrotor flight,
large margins in the $z$ direction can prevent unwanted effects due to
downwash~\cite{downwash:2017}. The minimum inter-agent distance during the
simulation, as measured from the centers of the agents, was $1.6\si{\meter}$.

\begin{figure}[h]
    \centering
    \includegraphics[width=.4\textwidth]{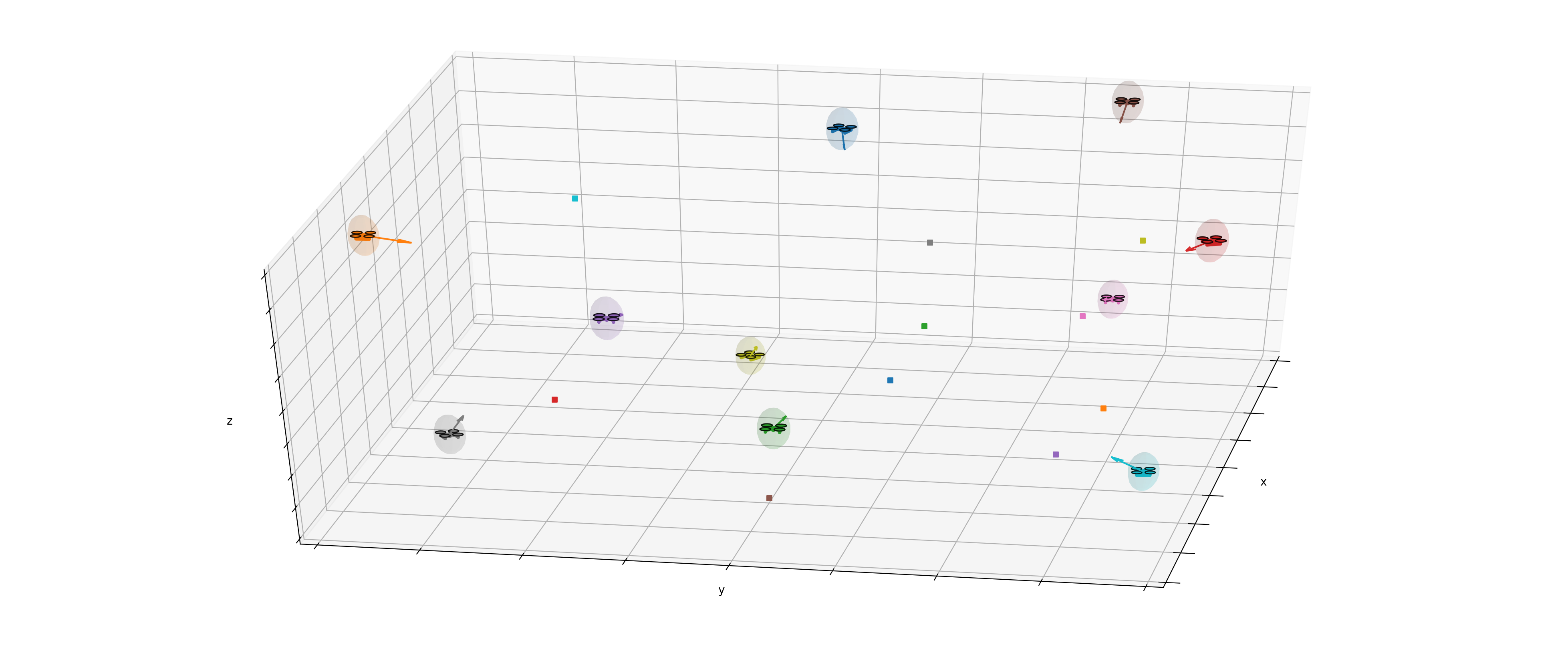}
    \includegraphics[width=.4\textwidth]{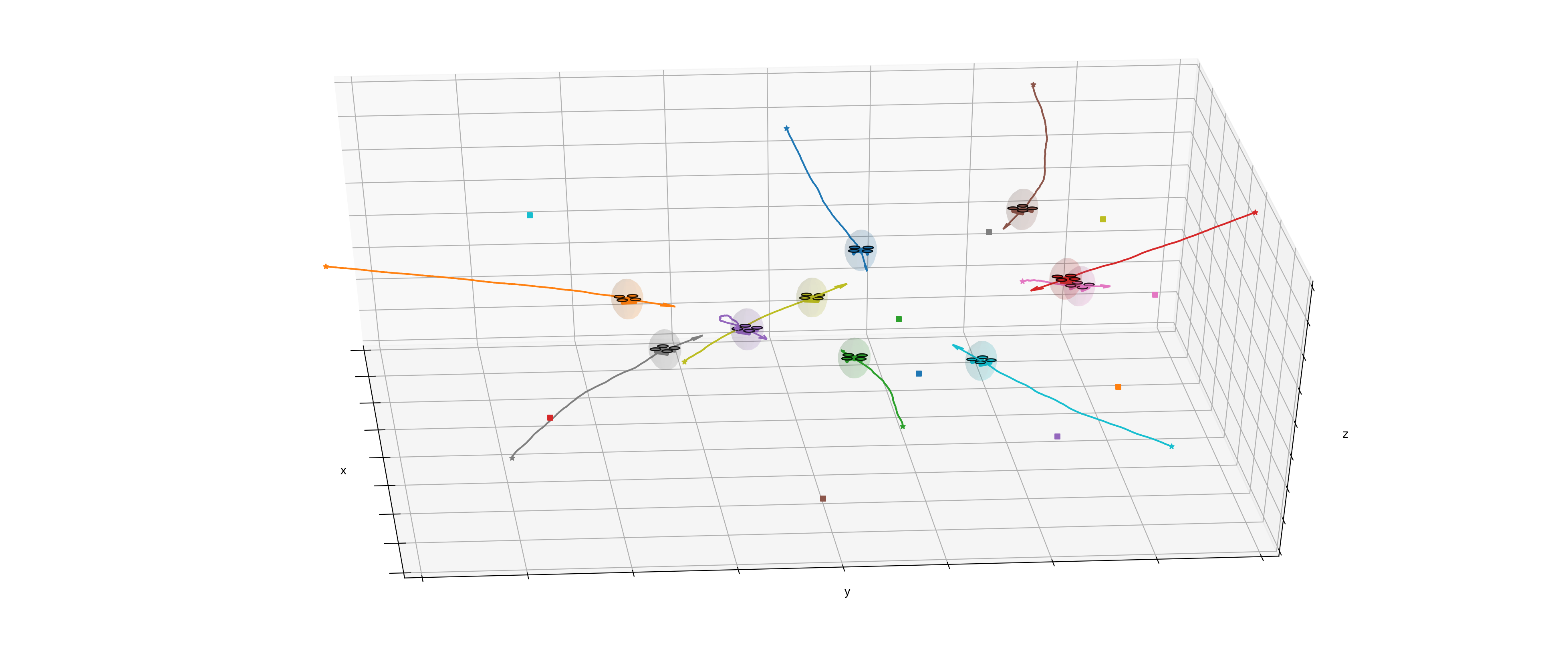}
    $t=0\hspace{2in} t=60$\\
    \includegraphics[width=.4\textwidth]{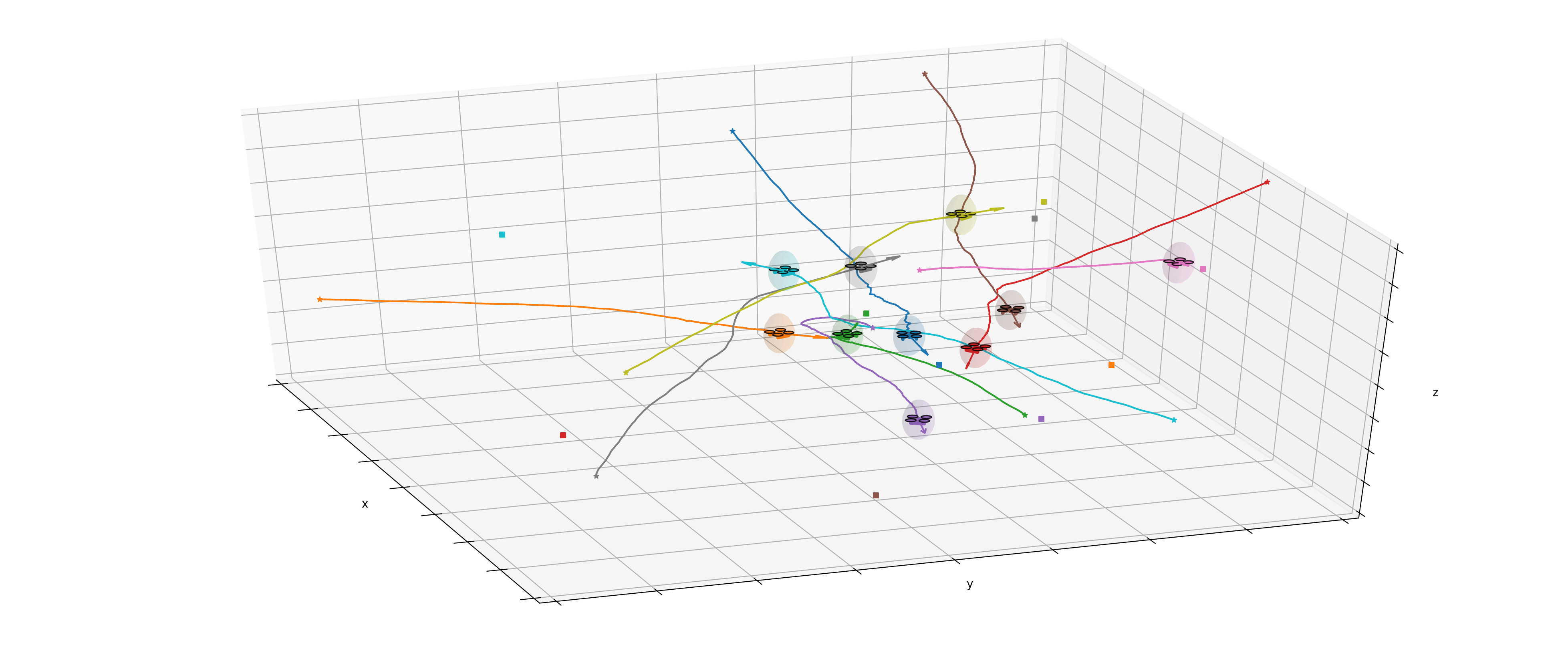}
    \includegraphics[width=.4\textwidth]{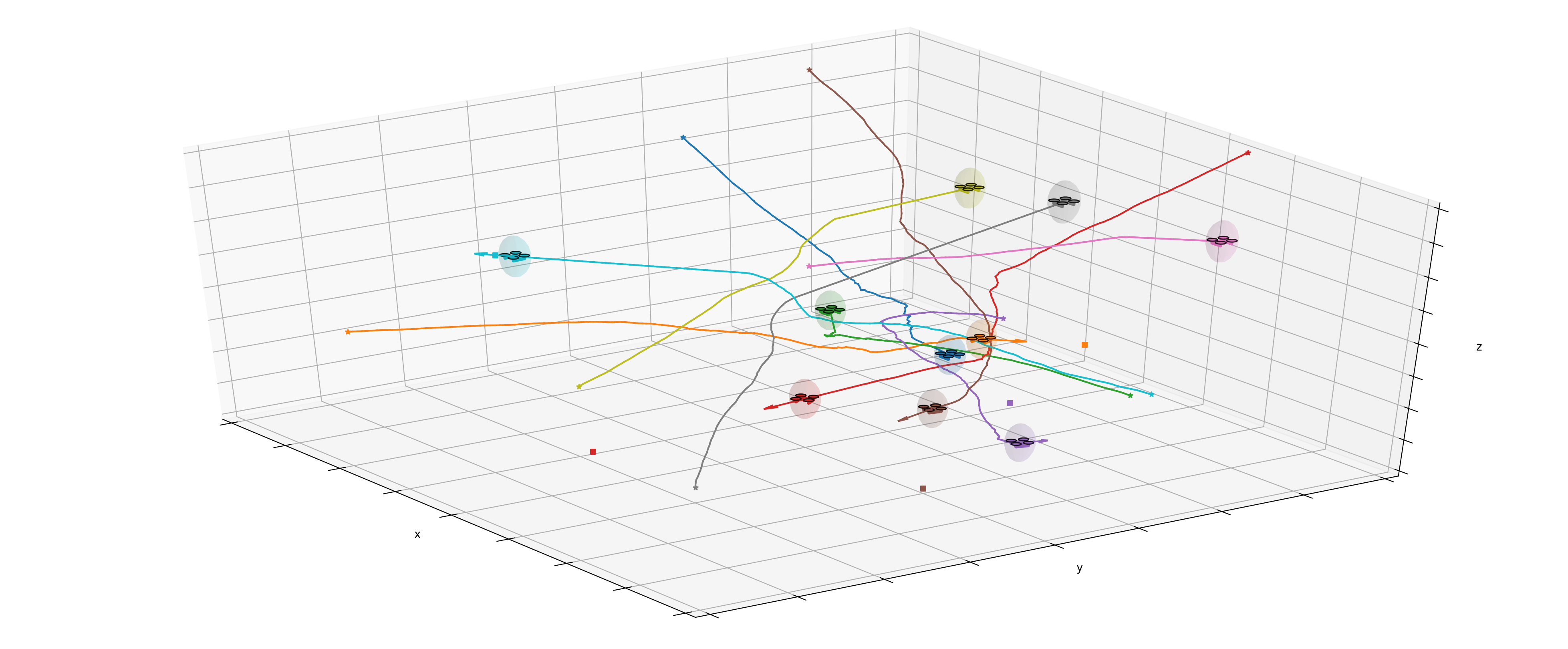}
    $t=120\hspace{2in} t=180$\\
    \includegraphics[width=.4\textwidth]{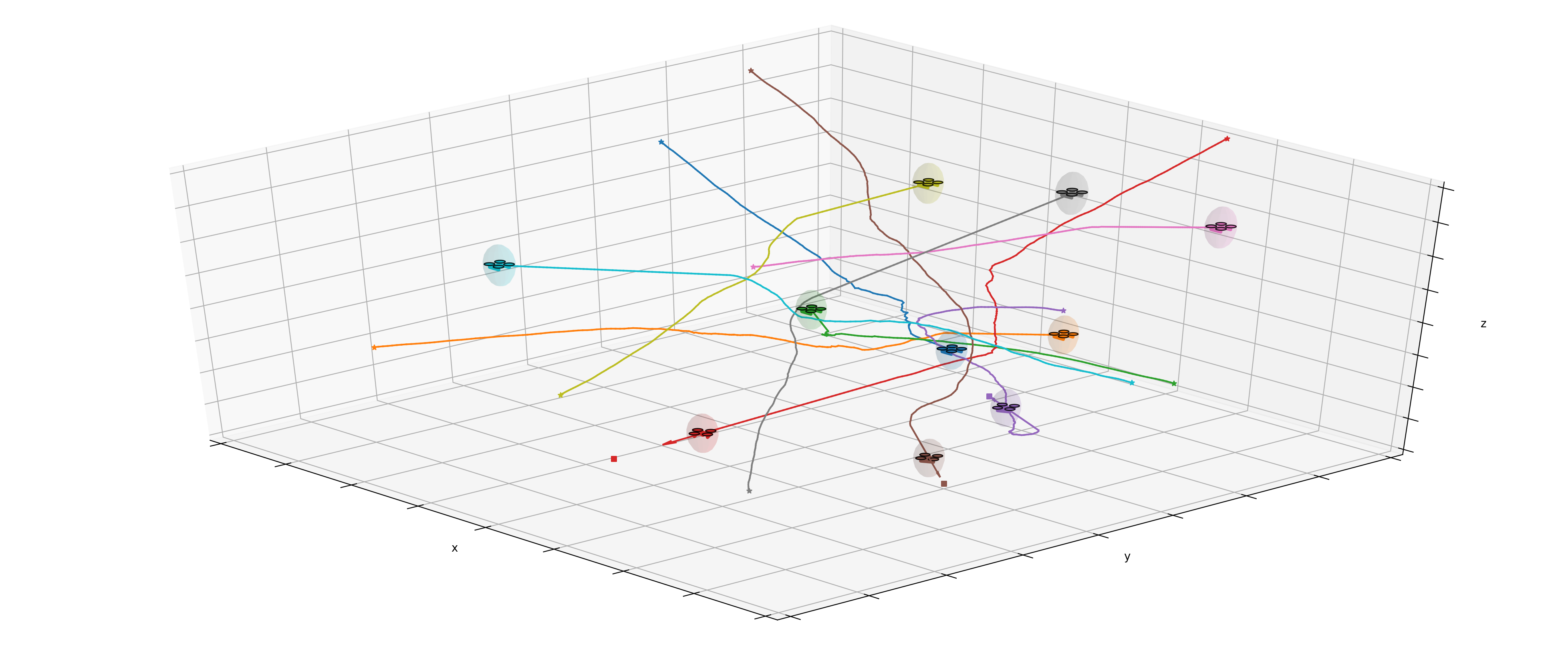}
    \includegraphics[width=.4\textwidth]{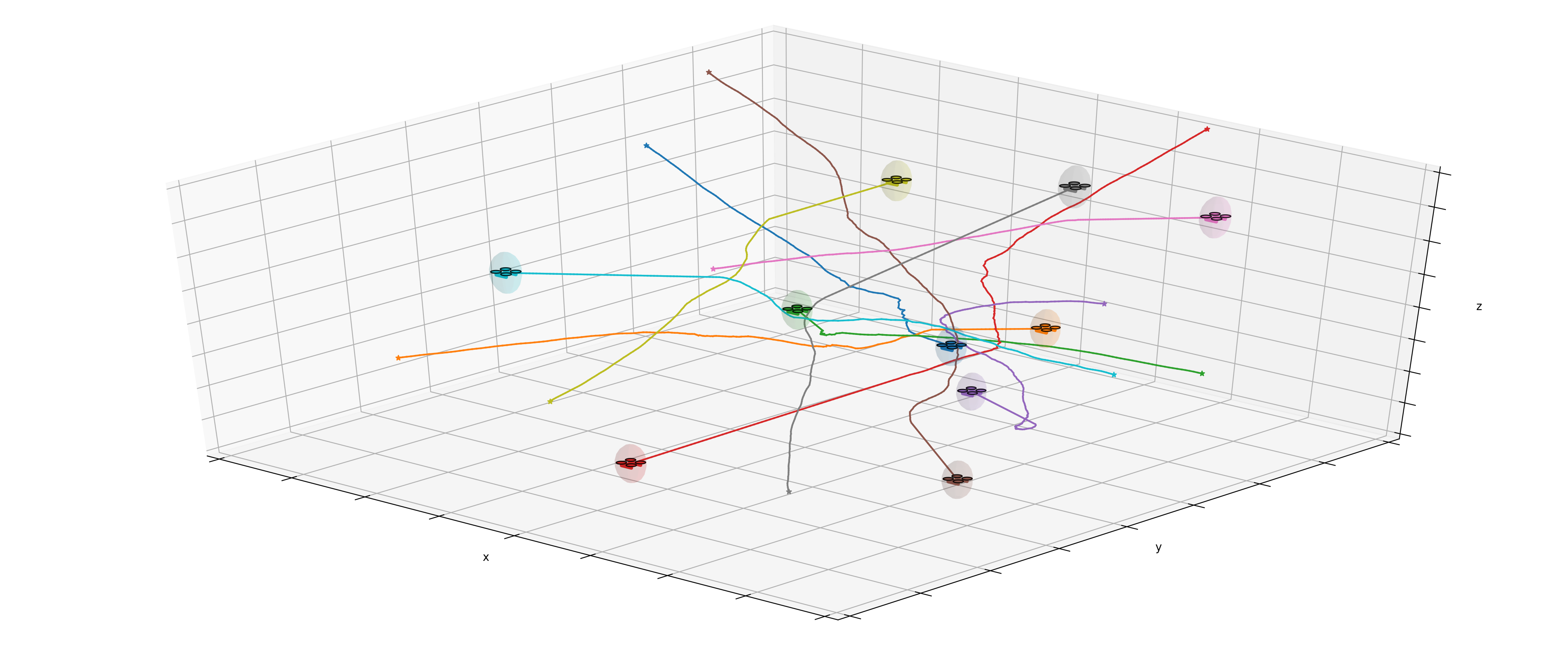}
    $t=210\hspace{2in} t=232$
    \vspace{.2in}
    \caption{Six time instances of a 3D simulation of 10 agents. Each agent 
    adds an ellipsoidal margin (shown) elongated in the z-axis to account for 
    downwash affects.}
    \label{fig::3dsim}
\end{figure}

\section{Closing Remarks} 
\vspace{-2mm}
While cloud and edge computing has relieved many of the burdens
distributed robotic systems encounter, network delays, disconnects,
and other failures are still commonplace at scale. Dependable and fast
algorithms that can run onboard are, therefore, critical for any
certifiable system. In this work, we presented a scalable system that
can work with simple or complex, distributed or centralized high level
planers to provide safe trajectories for a group of agents. Under the
assumptions stated, we showed that collision avoidance is guaranteed,
provided each agent follows this method. However, we observe practical
collision avoidance behavior even if only the ego agent follows this
method.  Computational performance results and simulations provide
evidence that this algorithm can potentially be used in
safety-critical applications for mobile robots with simple dynamics.

Future work will focus on creating a fast, easily extendable library for
automatically generating programs of the form of problem~\eqref{eq:final}, by
making use of the composition rules presented in the appendix. We suspect that
future support for warm starts and the ability to change parameters without
reconstructing the complete problem from scratch (as compared
to~\S\ref{sec:results:impl}) would yield a substantial speed up in solution
time. Additionally, we note that while immediately attempting to use
problem~\eqref{eq:final} in a Model Predictive Control (MPC) formalism yields a
nonconvex problem, we believe that there may be a approach to approximate
solution while retaining the feasibility properties assumed
in~\eqref{eq:proj}. We also plan on incorporating a global planner, which can 
handle static obstacles less conservatively than our
algorithm, into our trajectory planning pipeline.

\bibliographystyle{alpha}
\bibliography{cites}
\newpage

\section{Appendix}
\subsection{Dual functions}
In this subsection, we derive the Lagrange dual function for the Voronoi cell
generated by polyhedra.

\paragraph{Polyhedra.}
In the case where our set $\ellipse$ is specified by an affine constraint (\ie,
$\ellipse$ is a polyhedron),
\[
\ellipse = \{y ~|~ Ay \le b\},
\]
with $A \in \reals^{m\times n}$ and $b \in \reals^m$, then the dual function can be easily computed.

First, note that the minimizer of $\normsq{y} - 2c^Ty$ can be found by setting the gradient to zero (this is necessary and sufficient by convexity and differentiability), which yields that the optimal point is $y^* = c$, with optimal value
\[
\inf_y\left(\normsq{y} + 2c^Ty\right) = -\normsq{c}.
\]
Applying this result yields
\[
g(x, \lambda) = \inf_y \left(\normsq{y} - 2x^Ty + \lambda^T(Ay - b)\right) = - \normsq{x-A^T\lambda/2} - b^T\lambda,
\]
with $\lambda \ge 0$.

\subsection{Extensions}
There are several natural extensions to problem~(\ref{eq:main}). These extensions can be combined, as needed.

\paragraph{Union of convex sets.} We can extend the formalism of~(\ref{eq:main}) to include the Voronoi region generated by a finite union of convex sets. That is, if $S$ can be written in the form
\[
S = \bigcup_{i=1}^\ell ~\{y~|~f_i(y) \le 0\},
\]
for convex functions $f_i: \reals^n \to \reals^{m_i}$. Note that, in general,
$S$ will not be convex, while the resulting Voronoi region always is, since the
region is the intersection of a family of hyperplanes. Additionally, we will
require that there exist $y_i^0\in \reals^n$ with $f_i(y_i^0) < 0$ for $i=1,
\dots, \ell$ if $f_i$ is not affine. The corresponding problem is given by
\[
\begin{array}{ll}
\mbox{minimize}   & \normsq{x - \xgoal} \\
\mbox{subject to} & \normsq{\xcurr} - 2x^T\xcurr \le \inf_{f_i(y) \le 0} \left(\normsq{y} - 2y^Tx\right), ~~ i=1, \dots, \ell.
\end{array}
\]
In this case, we simply derive a dual function for each constraint with $i=1, \dots, \ell$, writing each as a constraint as in~(\ref{eq:dual-form}), each of which is convex.

\paragraph{Intersection of a convex set with the Voronoi cell.}
Given a convex set $C$ specified by
\[
C = \{x ~|~ h(x) \le 0\},
\]
where $h:\reals^n \to \reals^r$ is a convex function, then the problem
\[
\begin{array}{ll}
\mbox{minimize}   & \normsq{x - \xgoal} \\
\mbox{subject to} & \normsq{\xcurr} - 2x^T\xcurr \le \inf_{f(y) \le 0} \left(\normsq{y} - 2y^Tx\right)\\
& h(x) \le 0,
\end{array}
\]
is the intersection of the Voronoi cell in question and the set $C$, which is also easily solvable assuming $h$ can be easily evaluated at any point in its domain.

\end{document}